\documentclass{amsart}

\title[Segre classes on smooth projective toric varieties]{Segre classes on smooth\\ projective toric varieties}
\author{Torgunn Karoline Moe}
\address{Department of Mathematics, University of Oslo, P.O. Box 1053 Blindern, NO-0316 Oslo, NORWAY}
\email{torgunnk@math.uio.no}

\author{Nikolay Qviller}
\address{Centre of Mathematics for Applications, University of Oslo, P.O. Box 1053 Blindern, NO-0316 Oslo, NORWAY}
\email{nikolayq@cma.uio.no}



\date{\today}

\usepackage{mathrsfs}
\usepackage{graphicx}
\usepackage{hyperref}
\usepackage{amssymb}
\usepackage{moreverb}

\input xy
\xyoption{all}

\newtheorem{theorem}{Theorem}[section]
\newtheorem{lemma}[theorem]{Lemma}
\newtheorem{proposition}[theorem]{Proposition}

\theoremstyle{definition}

\newtheorem{defi}[theorem]{Definiton}
\newtheorem{example}[theorem]{Example}

\theoremstyle{remark}
\newtheorem{remark}[theorem]{Remark}

\numberwithin{equation}{section}

\newcommand{\Fe}{\ensuremath{\mathbb{F}_e\,}}

\begin{document}

\begin{abstract}
We provide a generalization of the algorithm of Eklund--Jost--Peterson for computing Segre classes of closed subschemes of projective $k$-space. The algorithm is here generalized to computing the Segre classes of closed subschemes of smooth projective toric varieties.
\end{abstract}

\maketitle

\setcounter{tocdepth}{1}

\tableofcontents

\section{Introduction}
\subsection{Background}
Segre classes are important objects appearing in intersection theory. Indeed, problems in for instance enumerative geometry frequently reduce to the computation of Segre classes. Nevertheless, the computation of the Segre class $s(Z,X)$ of a closed subscheme $Z$ of a scheme $X$, from the raw information contained in the sheaf of ideals $\mathscr{I}_{Z}$ alone, is a difficult problem. In terms of intersection theory, $s(Z,X)$ is a rational equivalence class of cycles supported on the scheme $Z,$ but it is often sufficient simply to know the push-forward $j_{\ast}s(Z,X) \in A_{\ast}(X),$ where $j: Z \hookrightarrow X$ denotes the inclusion map. If $Z$ has dimension $n,$ this push-forward has $n+1$ components $s_{i} \in A_{n-i}(X)$, for $0 \leq i \leq n$ (some readers may react to a slightly unconventional index notation; we give an explanation in Section \ref{segrenotation}).

Recently, Eklund, Jost and Peterson \cite{EKLUND} gave an algorithm computing these components $s_{i},$ when $Z$ is a closed subscheme of projective $k$-space over a field $\mathbb{K},$ with only input the ideal of $Z$ in the homogeneous coordinate ring $\mathbb{K}[x_{0},\ldots,x_{k}]$. A related, but slightly different algorithm for computing Segre classes of subschemes of projective spaces had already been provided by Aluffi (see \cite[Section 3]{Aluffi}). The algorithm of Eklund--Jost--Peterson is based on using the theory of residual intersections, by choosing a set of generic schemes envelopping the scheme $Z$, and then considering (and computing the degree of) the residual schemes of $Z$ in these. The result is a generic set of $n+1$ linear equations in the $s_{i},$ which is obviously enough to provide these classes. The method is then turned into a probabilistic computer algorithm by replacing the term ``generic'' with the term ``random.''

Our aim is to generalize the algorithm mentioned above to the case of ambient smooth projective toric varieties. As mentioned, Segre classes are complicated objects, but for subschemes of projective spaces they are quite manageable, mainly because $A^{\ast}(\mathbb{P}^{k})$ is simply $\mathbb{Z}[H]/(H^{k+1}),$ and also because these subschemes are easily described by homogeneous ideals in the coordinate ring. Now, $\mathbb{P}^{k}$ is a special instance of a $k$-dimensional toric variety, i.e., an irreducible algebraic variety $X$ containing a torus $T \cong (\mathbb{C}^{\ast})^{k}$ as a Zariski open subset, such that the action of $T$ on itself extends to an action of $T$ on $X.$ Generally, a normal toric variety of dimension $k$ can be seen to arise from a fan $\Sigma \subset \mathbb{R}^{k}$ (cf. \cite[Theorem 3.1.5]{COX2}), and many properties of the variety $X_{\Sigma}$ are reflected in the combinatorics of this fan. First, the generators and relations of the Chow ring of $X_{\Sigma}$ are fully described through the fan $\Sigma.$ Second, to a toric variety $X_{\Sigma}$ we can associate a homogeneous coordinate ring $\textrm{Cox}(X_{\Sigma})$, as introduced by Cox in \cite{COX}, which can also be described through the fan $\Sigma$. Moreover, closed subschemes $Z$ of $X_{\Sigma}$ are determined by homogeneous ideals in $\textrm{Cox}(X_{\Sigma}).$ The Chow ring and the homogeneous coordinate ring are particularly easy to formulate when the toric variety is smooth and projective, raising hope that we also in this situation have manageable Segre classes. Thus, the category of smooth projective toric varieties constitutes a natural place to look for a generalization of the Eklund--Jost--Peterson algorithm. In the following we will give an algorithm to compute the Segre classes of closed subschemes $Z$ of a smooth projective toric variety $X_{\Sigma}$, given only the fan $\Sigma$ of $X_{\Sigma}$ and the ideal $I \subset \textrm{Cox}(X_{\Sigma})$ of $Z$. 

\subsection{Structure} We start by recalling some basic facts about toric varieties and intersection theory in Section 2. The main theorem which forms the basis for the algorithm is based on residual intersection theory, and is proved in Section 3. Section 4 presents examples of computations of Segre classes of closed subschemes of smooth projective toric varieties. We give examples where the ambient toric variety is a Hirzebruch surface, $\mathbb{P}^{1} \times \mathbb{P}^{1} \times \mathbb{P}^{1},$ and a more general toric 3-fold. The algorithm is presented in Section 5.

\subsection{Conventions}\label{segrenotation}
 For classes in homology/cohomology theories, tradition imposes that a lower index $i$ denotes a component of dimension $i,$ while an upper index $i$ denotes codimension $i.$ However, Eklund--Jost--Peterson use the following convention, which we keep: For an $n$-dimensional subscheme $Z$ of a $k$-dimensional variety $X,$ $s_{i}$ denotes the pushforward to $X$ of the $(n-i)$-dimensional component of $s(Z,X).$ Thus, $s_{i}$ can be seen as a class in $A_{n-i}(X)$ or $A^{k-n+i}(X).$

\subsection{Acknowledgements}
We want to thank Ragni Piene for valuable supervision, Christine Jost for comments and for pointing us towards the \verb+Sage+ \cite{SAGE} implementation of intersection theory in the toric setting, and Kristian Ranestad for reminding us of the importance of nefness. Furthermore, we are grateful to Carel Faber and the referee for constructive remarks and comments. Finally, we want to thank Terje Kvernes at Drift and Georg Muntingh for computer assistance. All computations are performed using \verb+Macaulay2+ \cite{M2} by Grayson and Stillman with the module \verb+NormalToricVarieties+ by Smith, and \verb+Sage+ \cite{SAGE} by Stein et al.

\section{Intersection theory on toric varieties}

\subsection{Smooth projective toric varieties}
By a variety we here mean a reduced, irreducible scheme of finite type over $\mathbb{C}.$ A toric variety is a special kind of variety, and with a series of definitions, we will now make precise what we mean by a smooth projective toric variety.

\begin{defi}
A convex polyhedral cone in $\mathbb{R}^{k}$ is a set $\sigma = \left\{\sum_{v \in S_{\sigma}} \lambda_{v} v, \lambda_{v} \geq 0\right\},$ where $S_{\sigma} \subset \mathbb{R}^{k}$ is a finite subset. The cone is called rational if $S_{\sigma}$ is a finite subset of a lattice $N \cong \mathbb{Z}^{k}$. A rational convex polyhedral cone is strongly convex if $\sigma \cap (-\sigma) = \{0\}.$ 
\end{defi}

\begin{remark}
In the following we will take ``cone'' to mean ``strongly convex rational polyhedral cone''.
\end{remark}

\begin{defi}
A fan is a finite collection $\Sigma$ of cones in $N_{\mathbb{R}} := N \otimes_{\mathbb{Z}} \mathbb{R}$ such that for all $\sigma \in \Sigma$ and all faces $\tau$ of $\sigma$, we have $\tau \in \Sigma,$ and for all $\sigma, \tau$ in $\Sigma$, we have $\sigma \cap \tau \in \Sigma.$

A fan $\Sigma$ is smooth if every cone $\sigma$ in $\Sigma$ is smooth, i.e., the minimal generators of $\sigma$ form a part of a $\mathbb{Z}$-basis of the lattice $N$. 

A fan $\Sigma$ is complete if its support $|\Sigma|=\cup_{\sigma \in \Sigma}\sigma$ is all of $N_{\mathbb{R}}$. 
\end{defi}

To a fan $\Sigma$ one can associate an abstract toric variety $X_{\Sigma}$ by gluing the affine toric varieties $V_{\sigma} := \textnormal{Spec }\mathbb{C}[S_{\sigma}], \sigma \in \Sigma,$ ``along the fan structure'' (see \cite{COX2} for details). Here, $\mathbb{C}[S_{\sigma}]$ is the semigroup algebra $\mathbb{C}[\chi^{m}, m \in S_{\sigma}],$ with $\chi^{m}$ the character such that if $m=(b_{1}, \ldots, b_{k}),$ then $\chi^{m}(t) = \prod t_{i}^{b_{i}}.$ The fan $\Sigma$ is smooth if and only if $X_{\Sigma}$ is smooth, and complete if and only if $X_{\Sigma}$ is complete as an algebraic variety.

In the case where $\Sigma$ is the normal fan $\Sigma_{P}$ of a lattice polytope $P \subset \mathbb{R}^{k},$ we also get a line bundle $\mathscr{L}_{P}$ on $X_{P} := X_{\Sigma_{P}},$ which defines an embedding of $X_{P}$ into $\mathbb{P}^{n-1},$ where $n$ is the number of lattice points in $P.$ Thus, toric varieties whose fan is the normal fan of a lattice polytope are projective varieties. In particular, they are complete.

Consider a smooth projective $k$-dimensional toric variety $X_{\Sigma}$ arising from a fan $\Sigma$ in $\mathbb{R}^{k}.$ Let $r:=|\Sigma(1)|$ be the number of rays $\rho$ (one-dimensional cones) in $\Sigma.$ Let $D_{\rho}$ be the torus-invariant divisor associated to the ray $\rho$. Note that the torus-invariant divisors generate the group of Weil divisors; also, Cartier and Weil divisors coincide since $X$ is smooth, hence we only speak of \textit{divisors}. In \cite{COX}, Cox introduced the homogeneous coordinate ring of $X_{\Sigma}.$ This is the polynomial ring $S=\mathbb{C}[x_{\rho}|\rho \in \Sigma(1)]$ graded by the Abelian group $A^{1}(X)$ of divisors modulo linear equivalence. The grading is given in the following way: For a divisor $D$ on $X,$ there is a unique way to write $D = \sum \alpha_{\rho} D_{\rho}, \alpha_{\rho} \in \mathbb{Z}.$ Define $x^{D} := \prod_{\rho}x_{\rho}^{\alpha_{\rho}},$ and let $[D] \in A^{1}(X)$ be the degree of $x^{D}.$

\begin{remark}
Choosing a basis for $A^{1}(X)$ corresponds to grading the variables $x_{1},\ldots,x_{r}$ by $\mathbb{Z}^{r-k}.$ One can directly obtain admissible gradings as follows (see also \cite{MACLAGAN}): Let $\mathbf{b}_{1}, \ldots, \mathbf{b}_{r}$ denote the unique minimal lattice vectors generating the $r$ rays of the fan $\Sigma$ defining $X.$ Since $X$ is assumed to be projective (in particular complete), these vectors span $\mathbb{R}^{k},$ hence the map
$$
\mathbb{Z}^{k} \xrightarrow{[\mathbf{b}_{1} \ldots \mathbf{b}_{r}]^{T}}  \mathbb{Z}^{r}
$$
is injective, and we may fix an $((r-k)\times r)$-matrix $A=[\mathbf{a}_{1} \ldots \mathbf{a}_{r}]$ such that there is a short exact sequence
$$
0 \longrightarrow{} \mathbb{Z}^{k} \xrightarrow{[\mathbf{b}_{1} \ldots \mathbf{b}_{r}]^{T}}  \mathbb{Z}^{r} \xrightarrow{[\mathbf{a}_{1} \ldots \mathbf{a}_{r}]} \mathbb{Z}^{r-k} \longrightarrow{} 0. 
$$
In fact, $A$ is uniquely determined up to unimodular (determinant $\pm 1$) coordinate transformations of $\mathbb{Z}^{r-k}.$ The corresponding grading of $S$ is obtained by letting the degree of $x_{i}$ be the vector $\mathbf{a}_{i}$ for $1 \leq i \leq r.$
\end{remark}

\begin{defi}
Let $X_{\Sigma}$ and $S$ be as above. For each cone $\sigma \in \Sigma,$ let $\widehat{\sigma}$ be the divisor $\sum_{\rho \notin \sigma(1)} D_{\rho},$ and denote by $x^{\widehat{\sigma}}$ the monomial $\prod_{\rho \notin \sigma(1)}x_{\rho}$ in $S.$ The irrelevant ideal of $S$ is the ideal
\begin{equation*}
B := \langle x^{\widehat{\sigma}}, \sigma \in \Sigma \rangle \subset S.
\end{equation*}
An ideal $I \subset S$ is said to be $B$-saturated if $(I:B)=I.$ The saturation of $I$ by $B$ is the ideal
\begin{equation*}
(I:B^{\infty}) := \bigcup_{n \geq 1} (I:B^{n}),
\end{equation*}
which is equal to $(I:B^{n})$ for all $n \gg 0,$ since the sequence $$I \subseteq (I:B) \subseteq (I:B^{2}) \subseteq \ldots$$ is stationary. 
\end{defi}

The following proposition from \cite{COX} is essential to us.
\begin{proposition} 
Let $X_{\Sigma}, S$ and $B$ be as above. There is a 1-1 correspondence between closed subschemes of $X$ and graded $B$-saturated ideals of $S.$
\end{proposition}

\begin{proof}
This is Corollary 3.8, part (ii), in \cite{COX}.
\end{proof}

We use the same notation as Cox; $V(I)$ denotes the closed subscheme of $X$ associated to the ideal $I.$

\subsection{Intersection theory on toric varieties}
Recall that the Chow group of a $k$-dimensional variety $X$ is the additive group $A_{\ast}(X) := \bigoplus_{i=0}^{k}A_{i}(X)$ of cycles modulo rational equivalence. If $X$ is smooth, the Chow group can be endowed with a multiplication $\cdot$ such that, letting $A^{i}(X) := A_{k-i}(X)$ (i.e., grading by codimension), $A^{i}(X) \otimes_{\mathbb{Z}} A^{j}(X) \stackrel{\cdot}{\rightarrow} A^{i+j}(X).$ Let $A^{\ast}(X)$ denote $\bigoplus_{i=0}^{k}A^{i}(X);$ $(A^{\ast}(X),+,\cdot)$ is called the Chow ring of $X.$

We will now give a description of the Chow ring of a smooth projective toric variety $X_{\Sigma}$ using only the combinatorial information contained in the fan $\Sigma$. 

\begin{theorem}\label{thm:Danilov}
Let $X_{\Sigma}$ be a smooth projective $k$-dimensional toric variety over $\mathbb{C},$ and let $r := |\Sigma(1)|.$ Denote by $D_{1}, \ldots, D_{r}$ the torus-invariant divisors associated to the rays $\rho_1, \ldots, \rho_r$, generated by minimal lattice vectors $v_1, \ldots, v_r,$ and by $I$ the ideal in $\mathbb{Z}[D_{1},\ldots,D_{r}]$ generated by
\begin{enumerate}
\item all $D_{i_{1}}\cdot \ldots \cdot D_{i_{s}},$ where the $\{v_{i_{j}}\}_{1 \leq j \leq s}$ are not in a cone $\sigma \in \Sigma;$
\item all $\sum_{i=1}^{r} \langle \chi, v_{i}\rangle D_{i},$ for $\chi \in M,$ $M := N^{\vee} = \text{Hom}_{\mathbb{Z}}(N,\mathbb{Z})$ being the dual lattice of characters.
\end{enumerate}
Then $A^{\ast}(X) \cong \mathbb{Z}[D_{1},\ldots,D_{r}]/I$ as a $\mathbb{Z}$-algebra, and $A_{i}(X)$ has dimension
\begin{displaymath}
h_{i} := \textnormal{rk }A_{i}(X) = \sum_{j=i}^{k}(-1)^{j-i}{j \choose i}d_{k-j},
\end{displaymath}
where $d_{p}$ is the number of $p$-dimensional cones in $\Sigma.$
\end{theorem}

\begin{proof}
The first part is \cite[Proposition, p.106]{FULTONTORIC}, and the dimensional part is Theorem 10.8 in \cite{DANILOV} adapted to our notation.
\end{proof}

Next, we recall the definition of the Segre class of a subscheme of a scheme.
\begin{defi}
Given a scheme $X$ and a closed subscheme $Z$, where $Z$ is not an irreducible component of $X$, the Segre class $s(Z,X) \in A_{\ast}(X)$ of $Z$ in $X$ is defined by considering the blow-up $\widetilde{X}$ of $X$ along $Z;$ let $\widetilde{Z}$ be the exceptional divisor and $\pi: \widetilde{X} \rightarrow X$ the birational morphism obtained, then, if $\eta: \widetilde{Z} \rightarrow Z$ denotes the restriction of $\pi$ to $\widetilde{Z},$
\begin{equation*}
s(Z,X) := \sum_{i \geq 1}(-1)^{i-1}\eta_{\ast}(\widetilde{Z}^{i}) \in A_{\ast}(Z).
\end{equation*}

\end{defi}

\begin{remark}
Segre classes are birational invariants; if $\eta: \widetilde{X} \rightarrow X$ is a birational morphism, $Z \subset X$ is a subscheme and $\widetilde{Z} = \eta^{-1}(Z)$ is the inverse image scheme, then $s(Z,X)$ is equal to $\eta_{\ast}s(\widetilde{Z},\widetilde{X})$ (cf. \cite[Proposition 4.2]{FULTON}).
\end{remark}

\begin{defi}
Suppose $X$ has pure dimension $k$ and $j: Z \hookrightarrow X$ is a not necessarily pure-dimensional subscheme of dimension $n.$ Then $s(Z,X) \in A_{\ast}(Z),$ as well as its push-forward $j_{\ast}s(Z,X) \in A_{\ast}(X),$ has components in dimensions 0 to $n.$ We define $s_{i},$ for each $0 \leq i \leq n,$ to be the push-forward of the $(n-i)$-dimensional component of $s(Z,X).$ Thus, $s_{i}$ is a class in $A_{n-i}(X) \cong A^{k-n+i}(X).$ 
\end{defi}

Suppose $D \subset W \subset V$ are closed embeddings of schemes, with $D$ a Cartier divisor on $W.$ Then there is a closed subscheme $R$ of $W,$ called the \textit{residual scheme} to $D$ in $W,$ such that $\mathscr{I}_{W} = \mathscr{I}_{D}\cdot \mathscr{I}_{R}.$ Furthermore, the Segre classes of $D,W$ and $R$ in $V$ are related by the following proposition:

\begin{proposition}\cite[Proposition 9.2]{FULTON}
Suppose $D \subset W \subset V$ are closed embeddings, with $V$ a $k$-dimensional variety and $D$ a Cartier divisor on $V.$ If $R$ is the residual scheme to $D$ in $W,$ then for all $m,$
\begin{equation*}
s(W,V)_{m} = s(D,V)_{m} + \sum_{j=0}^{k-m} {k-m \choose j} (-D)^{j}s(R,V)_{m+j} \in A_{m}(W).
\end{equation*}
\end{proposition}

\begin{proof}
See \cite{FULTON}. Note that there is a slight abuse of notation here; the components of $s(D,V)$ and $s(R,V)$ are implicitly pushed forward to $W.$
\end{proof}

This result, which may seem innocent at first sight, is one of the few statements which make it possible to work with Segre classes of ``complicated'' subschemes. The fact that Segre classes are invariant under birational morphisms makes it all the more powerful. Indeed, the general situation, where $Z \subset W \subset V$ are closed embeddings, but $Z$ is not a Cartier divisor on $V,$ can be reduced to the previous one by blowing up $V$ along $Z.$ Furthermore, pushing this observation a bit further allows the splitting-up of an intersection product into parts supported on a ``known'' subscheme $Z$ and a suitable ``remainder'' scheme $R:$

\begin{proposition}\label{GeneralRIF}
Let $V$ be a smooth $k$-dimensional variety and suppose that for $1 \leq i \leq d,$ $X_{i}$ is an effective divisor on $V.$ If $Z$ is a closed subscheme of the scheme-theoretic intersection $W:=X_{1} \cap \ldots \cap X_{d},$ then, letting $N$ denote $\bigoplus_{i=1}^{d}N_{i}|Z,$ with $N_{i}$ the normal bundle of $X_{i}$ in $V,$ we have
\begin{equation*}
X_{1} \cdot \ldots \cdot X_{d} = \left\{c(N) \cap s(Z,V)\right\}_{k-d} + \mathbb{R} \in A_{\ast}(W),
\end{equation*}
where the residual class $\mathbb{R}$ is defined as follows: Let $\pi: \widetilde{V} \rightarrow V$ be the blow-up of $V$ along $Z,$ let $\widetilde{Z}$ be the exceptional divisor and $\widetilde{R}$ the residual scheme of $\widetilde{Z}$ in $\pi^{-1}(W).$ Also, let $R := \pi(\widetilde{R})$ and denote by $\eta$ the induced morphism from $\widetilde{R}$ to $R.$ Then, with $\mathscr{O}(-\widetilde{Z})$ the pullback to $\widetilde{Z}$ of $\mathscr{O}_{\widetilde{V}}(-\widetilde{Z})$ and $\widetilde{N}$ the pullback to $\widetilde{Z}$ of $N,$ we define
\begin{equation*}
\mathbb{R} := \eta_{\ast}\left\{c(\widetilde{N} \otimes \mathscr{O}(-\widetilde{Z}) ) \cap  s(\widetilde{R},\widetilde{V})\right\}_{k-d} \in A_{\ast}(R).
\end{equation*}

\end{proposition}

\begin{proof}
This is a special case of \cite[Corollary 9.2.3]{FULTON}, applied to the diagram
\[
\xymatrix
{
& R  \ar@{^{(}->}[]+<0ex,-2ex>;[d] & \\
Z \; \ar@{^{(}->}[r] & W = \bigcap_{i=1}^{d} X_{i} \; \ar@{^{(}->}[r] \ar@{^{(}->}+<0ex,-2.5ex>;[d] & V \ar@{^{(}->}+<0ex,-2.5ex>;[d]^{\delta} \\
& X = X_{1} \times \ldots \times X_{d} \; \ar@{^{(}->}[r] & V \times \ldots \times V
}
\]
where $\delta$ is the diagonal embedding. Indeed, $X \cdot V = X_{1} \cdot \ldots \cdot X_{d} \in A_{\ast}(W).$

Note that we once again are implicitly pushing both classes (one supported on $Z$ and one on $R$) forward to $W.$
\end{proof}

\begin{remark}\label{res}
The residual scheme $R$, in Fulton's sense, is always defined relatively to a Cartier divisor $D$ lying inside a larger scheme, $W.$ While one can always reduce to this situation by blowing-up, alternative and more direct approaches have been suggested. As Fulton points out \cite[Example 9.2.8]{FULTON}, Peskine--Szpiro define the residual scheme of $Z$ in $W$ by taking its sheaf of ideals to be the ideal of functions multiplying all functions in $\mathscr{I}(Z)$ into elements of $\mathscr{I}(W).$ That is,
\begin{equation*}
\mathscr{I}(R)/\mathscr{I}(Z) = \textnormal{Hom}_{\mathscr{O}_{V}}(\mathscr{O}_{Z},\mathscr{O}_{W}).
\end{equation*}
The construction can be found in \cite{PESKINE}. Note that if $Z$ \textit{is} a Cartier divisor we recover the same residual scheme $R$ as defined previously. On the other hand, if the ambient scheme $V$ is projective space $\mathbb{P}^{k}_{\mathbb{K}},$ then closed subschemes are described by homogeneous ideals in the coordinate ring $\mathbb{K}[x_{0},\ldots,x_{k}],$ not containing the irrelevant ideal $(x_{0},\ldots, x_{k}).$ In \cite{EKLUND}, Eklund, Jost and Peterson consider the residual scheme $R$ defined by the saturation ideal $(I_{W}:I_{Z}^{\infty})$ (equal to Peskine--Szpiro's residual scheme if $W$ is reduced along $Z$), and show that the residual class $\mathbb{R}$ from Proposition \ref{GeneralRIF} is equal to $[R],$ provided $W$ is ``sufficiently generic'' outside $Z.$ This gives an answer to the question posed by Fulton in \cite[Example 9.2.8]{FULTON} in the case of projective spaces, and is here extended to the case of toric varieties.
\end{remark}

This motivates the following definition in the setting of toric varieties.

\begin{defi}
Let $X_{\Sigma}$ be a smooth toric variety with Cox ring $S$ and irrelevant ideal $B,$ and let $Z \subset W$ be closed subschemes determined by $B$-saturated ideals $I_{Z}, I_{W}$ in $S.$ We define the residual scheme $R$ to $Z$ in $W$ to be the closed subscheme of $X_{\Sigma}$ determined by the saturated ideal $(I_{W}:I_{Z}^{\infty}).$ 
\end{defi}

In the next section, we propose to show that for $W$ ``sufficiently generic'' outside $Z$ (in a sense to be made precise), the class of $R$ is equal, for all smooth, projective toric varieties, to the residual class $\mathbb{R}$ as defined by Fulton (cf. Remark \ref{res}).

\section{A recursive formula for Segre classes}
The aim is to generalize the procedure of Eklund--Jost--Peterson for computing Segre classes of closed subschemes for a general $k$-dimensional smooth projective toric variety $X := X_{\Sigma}$ associated to a fan $\Sigma.$ Let $S := \mathbb{C}[x_{\rho} | \rho \in \Sigma(1)]$ be the Cox ring of $X,$ and denote by $B$ the irrelevant ideal. Suppose $I \subseteq S$ is a graded $B$-saturated ideal, giving a closed subscheme $j: Z \hookrightarrow X$ of dimension $n.$ We want to find $j_{\ast}s(Z,X) = s_{0} + \ldots + s_{n},$ with $s_{i} \in A^{k-n+i}(X).$ Choose a basis for $A^{1}(X) \cong \textnormal{Pic}(X)$ (we have an isomorphism since $X$ is smooth), and thereby a $\mathbb{Z}^{r-k}$-grading of $S.$ Suppose $I = (g_{0},\ldots,g_{t})$ where each $g_{j}$ is a homogeneous generator (with respect to chosen grading).

\begin{defi}
Let $\textnormal{Nef}(X)_{\mathbb{Z}} \subseteq \textnormal{Pic}(X)$ be the integral nef cone (i.e., the intersection of the real nef cone with $\textnormal{Pic}(X)$), and let $N(g)$ denote the cone $N(g):=\left(\textnormal{Nef}(X)_{\mathbb{Z}} + \textnormal{deg }g\right)$. Moreover, let $N(g_{0},\ldots,g_{t})$ denote the cone $$N(g_{0},\ldots,g_{t}) := \bigcap_{i=0}^{t} N(g_i).$$
\end{defi}

\begin{lemma}
The cone $N(g_{0},\ldots,g_{t})$ is non-empty.
\end{lemma}

\begin{proof}
Since $X$ is projective, $\textnormal{Nef}(X)_{\mathbb{Z}}$ is a full-dimensional cone \cite[Theorem 6.3.22]{COX2}. Since each $N(g_i)$ is simply a translated copy of $\textnormal{Nef}(X)_{\mathbb{Z}},$ the lemma follows.
\end{proof}

Let $\alpha \in \textnormal{Pic}(X)$ be any element of $N(g_{0},\ldots,g_{t}).$ Thus, $\alpha - \textnormal{deg }g_{i}$ is nef (or equivalently, the associated linear system is base-point free) for all $0 \leq i \leq t.$ Then, choose $k$ general elements $f_{1}, \ldots, f_{k} \in I(\alpha) := I \cap S_{\alpha}.$ For each $d \in \{k-n, \ldots, k\},$ let $J_{d}$ denote the ideal $((f_{1},\ldots, f_{d}):B^{\infty}) \subseteq S.$ Then $Z$ is a subscheme of $V(J_{d}),$ and we may consider its residual scheme $R_{d} := V((J_{d}:I^{\infty})).$

\begin{theorem}\label{thm:main}
Let $X$ be a $k$-dimensional smooth projective toric variety, $Z$ a closed $n$-dimensional subscheme of $X$, which is not an irreducible component of $X.$ For all $k-n \leq d \leq k,$ the residual scheme $R_{d}$ is either empty or has pure dimension $k-d,$ and the following equality holds in the Chow ring of $X$:
\begin{equation*}
[V(f_{1})] \cdot \ldots \cdot [V(f_{d})] = \left\{j_{\ast}\left(c(N(d)) \cap s(Z,X)\right)\right\}_{k-d} + [R_{d}],
\end{equation*}
where $N(d)$ is the bundle $\bigoplus_{i=1}^{d}N_{V(f_{i})}X|Z.$ 
\end{theorem}

\begin{proof}
The analogous theorem was proved for projective spaces by Eklund, Jost and Peterson \cite[Theorem 3.2]{EKLUND}. Only small, formal modifications are needed to generalize it for toric varieties. For sake of completeness, we include the modified proof here.

Since each $|\alpha - \textnormal{deg } g_{i}|$ is base-point free, the ideal $\langle I(\alpha) \rangle$ generated by $I(\alpha) = I \cap S(\alpha)$ defines the same subscheme as $I.$ Indeed, we get $(\langle I(\alpha) \rangle :B^{\infty}) = (I:B^{\infty});$ to see this, consider the following procedure in $t+1$ steps. First, replace the generator $g_{0}$ by all products of $g_{0}$ with monomials in $S$ of degree $\alpha - \textnormal{deg }g_{0}.$ Since the corresponding linear system is base-point free and we consider \textit{all} monomials, the saturation of the new ideal by $B$ is equal to $(I:B^{\infty}).$ Now, do the same with each $g_{i}, 1 \leq i \leq t.$ We may therefore assume that the generators $g_{i}$ of $I$ are homogeneous of same degree $\alpha.$

Consider the blow-up $\pi: \widetilde{X} \rightarrow X$ of $X$ along the closed subscheme $Z = V((g_{0},\ldots,g_{t})),$ and let $\widetilde{Z}$ be the exceptional divisor. Note that by \cite[Proposition II.7.16]{Hart:1977}, $\widetilde{X}$ is a variety. By \cite[Section 4.4]{FULTON}, the line bundle $\pi^{\ast}\mathscr{O}_{X}(\alpha) \otimes \mathscr{O}_{\widetilde{X}}(-\widetilde{Z})$ defines an embedding of $\widetilde{X}$ as a subvariety of the toric variety $X \times \mathbb{P}^{t},$ whose Cox ring is $S':=S[y_{0},\ldots,y_{t}].$ The ideal of $\widetilde{X}$ in $S'$ contains $(g_{i}y_{j}-g_{j}y_{i}, 0 \leq i < j \leq t) \subseteq S'.$ Denote by $\phi: \widetilde{X} \rightarrow \mathbb{P}^{t}$ the composition of the embedding $\widetilde{X} \hookrightarrow X \times \mathbb{P}^{t}$ and the projection map $X \times \mathbb{P}^{t} \rightarrow \mathbb{P}^{t}.$

Put $W_d:= V((f_{1}, \ldots, f_{d}):B^{\infty}),$ and let $\widetilde{W}_{d} := \pi^{-1}(W_{d}).$ Then $\widetilde{Z}$ is a divisor contained in $\widetilde{W}_{d},$ and its residual scheme $\widetilde{R}_{d}$ in $\widetilde{W}_{d}$ has ideal $(I_{\widetilde{W}_{d}}:I_{\widetilde{Z}}^{\infty}).$ Working locally on $\widetilde{X}$ we will show that $\widetilde{R}_{d}$ is either empty or of pure dimension $k-d.$ Consider the affine open set $U := U_{\rho\beta} \subset \widetilde{X}$ where $x_{\rho} \neq 0 \neq y_{\beta}.$ Letting $w_{i} := y_{i}/y_{\beta},$ we have coordinates $w_{0},\ldots, \widehat{w_{\beta}},\ldots, w_{t}$ on an affine $\mathbb{C}^{t} \subset \mathbb{P}^{t}.$ By definition, there exist general vectors $(\lambda^{0}_{i},\ldots, \lambda^{t}_{i}) \in \mathbb{C}^{t+1}$ such that for each $i,$
\begin{displaymath}
f_{i} = \sum_{j=0}^{t}\lambda^{j}_{i}g_{j}.
\end{displaymath}

Since $g_{i}y_{j} - g_{j}y_{i}= 0$ for all $i,j,$ we get $g_{j}y_{\beta} = g_{\beta}y_{j}.$ Dividing both sides by $y_{\beta}$ yields $g_{j} = g_{\beta}w_{j}.$ But $\widetilde{Z} = V((g_{0}, \ldots, g_{t})),$ so $\widetilde{Z} \cap U$ is defined by $g_{\beta}.$ Therefore, we also get
\begin{displaymath}
f_{i} = \left(\sum_{j=0}^{t}\lambda^{j}_{i}w_{j}\right)g_{\beta},
\end{displaymath}
and consequently the ideal of $\widetilde{R}_{d} \cap U$ is $\left(\sum_{j=0}^{t}\lambda^{j}_{1}w_{j},\ldots,\sum_{j=0}^{t}\lambda^{j}_{d}w_{j}\right).$ So $\widetilde{R}_{d} = \phi^{-1}(L)$ for a general linear subspace $L \subset \mathbb{P}^{t}$ of codimension $d,$ and is either empty or has pure dimension $k-d,$ and $\widetilde{R}_{d} \cap \widetilde{Z}$ has pure dimension $k-d-1$ unless it is empty. Thus, no irreducible component of $\widetilde{R}_{d}$ can be contained in $\widetilde{Z}.$ Now, since $\pi^{\ast}\mathscr{O}_{X}(\alpha) \otimes \mathscr{O}_{\widetilde{V}}(-\widetilde{Z})$ is base-point free on the variety $\widetilde{X}$ and $\widetilde{R}_d$ is cut out by generic sections, the embedding of $\widetilde{R}_d$ in $\widetilde{X}$ is also regular (cf. \cite[Lemma 2.1]{EKLUND}).

Using the residual intersection formula (Proposition \ref{GeneralRIF}), we see that
\begin{equation*}\label{eqn:essential_equality}
[V(f_{1})] \cdot \ldots \cdot [V(f_{d})] = \left\{c(N(d)) \cap s(Z,X)\right\}_{k-d} + \mathbb{R}_{d},
\end{equation*}
in $A_{\ast}(W_{d}).$ Here, $\mathbb{R}_{d} := \eta_{\ast}\left(\left\{c(\widetilde{N(d)} \otimes \mathscr{O}(-\widetilde{Z}) ) \cap s(\widetilde{R}_d,\widetilde{X})\right\}_{k-d}\right),$ letting $\widetilde{N(d)}$ denote the pullback to $\widetilde{Z}$ of $N(d).$ Now, because $\widetilde{R}_{d}$ has pure dimension $k-d,$ the class $\mathbb{R}_{d}$ is actually equal to $\eta_{\ast}[\widetilde{R}_{d}],$ with $[\widetilde{R}_{d}] \in A_{\ast}(\widetilde{R}_{d})$ the fundamental class of $\widetilde{R}_{d}.$ Indeed, the embedding of $\widetilde{R}_{d}$ in $\widetilde{X}$ being regular, we have $s(\widetilde{R}_{d},\widetilde{X}) = c(N_{\widetilde{R}_{d}}\widetilde{X})^{-1} \cap [\widetilde{R}_{d}].$ Consequently, $c(\widetilde{N(d)} \otimes \mathscr{O}(-\widetilde{Z}) ) \cap s(\widetilde{R}_d,\widetilde{X}) = [\widetilde{R}_{d}] + \ldots,$ where the remaining terms have strictly lower dimension and therefore vanish when taking the part of dimension $k-d.$

We now push the entire equality (\ref{eqn:essential_equality}) forward to $X.$ Thus, what remains to show is that $\pi_{\ast}[\widetilde{R}_{d}] = [R_{d}],$ where $[\widetilde{R}_{d}]$ is now the class of $\widetilde{R}_{d}$ in $A_{\ast}(\widetilde{X}).$ Now, by definition $\widetilde{R}_{d}$ is supported on $\overline{\widetilde{W}_{d} \setminus \widetilde{Z}}$ and $R_{d}$ on $\overline{W_{d} \setminus Z}.$ On the other hand, the ideal of $R_{d}$ is $(I_{W_{d}}:I_{Z}^{\infty}).$ If $I_{W_{d}}$ has primary decomposition $\bigcap_{i}Q_{i},$ then $I_{R_{d}}$ has primary decomposition
\begin{equation*}
\bigcap_{V(Q_{i}) \not\subseteq Z} Q_{i}.
\end{equation*}
Since $\pi$ induces an isomorphism $(\widetilde{W}_{d} \setminus \widetilde{Z}) \cong (W_{d} \setminus Z)$ and no irreducible component of $\widetilde{R}_{d}$ is contained in $\widetilde{Z},$ it follows that $\pi_{\ast}[\widetilde{R}_{d}] = [\pi(\widetilde{R}_{d})] = [R_{d}].$
\end{proof}

\begin{proposition}\label{prop:formula}
Let $Z \stackrel{j}{\hookrightarrow} X_{\Sigma}$ be the closed subscheme of a smooth $k$-dimensional toric variety $X_{\Sigma}$ associated to a fan $\Sigma$ in $\mathbb{R}^{k},$ determined by an ideal $I=(g_{0},\ldots,g_{t})$ in the homogeneous coordinate ring $S_{\Sigma}.$ Let $n := \textnormal{dim }Z.$ For generic polynomials $f_{1}, \ldots, f_{k}$ in $I(\alpha),$ $\alpha$ being an element of $N(g_{0},\ldots,g_{t}),$ and for all $d \in \{k-n, \ldots, k\},$ let $R_{d}$ be the scheme associated to the ideal
\begin{equation*}
(((f_{1}, \ldots, f_{d}):B^{\infty}):I^{\infty}). 
\end{equation*}
Letting $j_{\ast}s(Z,X) = s_{0} + \ldots + s_{n}$ with $s_{i} \in A^{k-n+i}(X),$ there is a recursive formula for the $s_{i}:$
\begin{eqnarray*}
s_{0} & = & \alpha^{k-n} - [R_{k-n}] \\
s_{i} & = & \alpha^{i+k-n} - [R_{i+k-n}] - \sum_{j=0}^{i-1} {i+k-n \choose i-j}\alpha^{i-j}s_{j}, \text{ for all } i \geq 1.
\end{eqnarray*}
\end{proposition}

\begin{proof}
Note that the equality
\begin{equation*}
[V(f_{1})] \cdot \ldots \cdot [V(f_{d})] = \left\{j_{\ast} \left(c(N(d)) \cap s(Z,X)\right)\right\}_{k-d} + [R_{d}]
\end{equation*}
simplifies. Indeed, since the degree (in the Cox ring) of each $f_{i}$ is $\alpha,$ the left hand side is simply $\alpha^{d}.$ On the other hand, $N(d) = j^{\ast}E_{d},$ where $E_{d} := \bigoplus_{i=1}^{d} \mathscr{O}_{X}(\alpha),$ so the right hand side becomes
\begin{equation*}
\left\{c(E_{d}) \cap j_{\ast}s(Z,X)\right\}_{k-d} + [R_{d}] = \left\{(1+\alpha)^{d} \cap \sum_{i=0}^{n} s_{i} \right\}_{k-d} + [R_{d}]. 
\end{equation*}
Hence the relation becomes: 
\begin{eqnarray*}
\alpha^{d} & = & \left\{\sum_{j=0}^{d}{d \choose j}\alpha^{j}\left(s_{0} + s_{1} + \ldots + s_{n}\right)\right\}_{k-d} + [R_{d}] \\
\Longleftrightarrow \alpha^{d} & = & \sum_{i+j = d-(k-n)} {d \choose j}\alpha^{j}s_{i} + [R_{d}]
\end{eqnarray*}
Taking $d=k-n$ yields $s_{0} = \alpha^{k-n} - [R_{k-n}],$ and it is easily seen that the relation
\begin{equation*}
s_{i} = \alpha^{i+k-n} - [R_{i+k-n}] - \sum_{j=0}^{i-1} {i+k-n \choose i-j}\alpha^{i-j}s_{j}
\end{equation*}
holds for all $i \geq 1.$
\end{proof}

To explicitly find the $s_{i}$ from these equations, one needs a way to compute each $[R_{d}].$ Recall that $R_{d}$ is pure-dimensional and has codimension $d$ in $X.$ Hence its class can be put in the form
\begin{equation*}
[R_{d}] = \sum_{i=1}^{h_{k-d}} b^{(d)}_{i}\omega^{(d)}_{i}, b^{(d)}_{i} \in \mathbb{Z}.
\end{equation*}
Here, the $\omega^{(d)}_{i}$ are generators of $A^{d}(X)$ and $h_{i} = \textnormal{rk }A_{i}(X).$ For each set $\underline{p}$ of non-negative integers $p_{j}, 1 \leq j \leq r,$ such that $\sum p_{j} = k-d,$ consider the class
\begin{equation*}
[R_{d}] \cdot \prod_{j=1}^{r}D_{j}^{p_{j}} =: \gamma_{\underline{p}}\omega^{(k)} \in A^{k}(X),
\end{equation*}
with $\{\omega^{(k)}\}$ the basis for $A^{k}(X).$ On the other hand,
\begin{equation*}
[R_{d}] \cdot \prod_{j=1}^{r}D_{j}^{p_{j}} = \sum_{i=1}^{h_{k-d}}b^{(d)}_{i} \underbrace{\left(\omega^{(d)}_{i}\prod_{j=1}^{r}D_{j}^{p_{j}}\right)}_{\beta_{i,\underline{p}}\omega^{(k)}},
\end{equation*}
so we have a set of linear equations in the $b^{(d)}_{i},$ namely
\begin{equation*}
\left\{\sum_{i=1}^{h_{k-d}}b^{(d)}_{i}\beta_{i,\underline{p}} = \gamma_{\underline{p}}, \quad \underline{p} \in \Gamma_{k-d}\right\},
\end{equation*}
where $\Gamma_{k-d}$ is the set of all $r$-tuples of non-negative integers whose sum is $k-d.$ Note that many of these equations will simply state that $0=0,$ or be linear combinations of each other. In fact, it would have been sufficient to consider the intersection products of $[R_{d}]$ with the generators of $A^{k-d}(X),$ but this is not implemented in the ToricVariety module of \verb+Sage+ \cite{SAGE}.

Intersecting $[R_{d}]$ with $D_{1}$ a total of $p_{1}$ times, $D_{2}$ a total of $p_{2}$ times, etc., corresponds to adding $p_{1}$ generic polynomials of degree $[D_{1}],$ $p_{2}$ generic polynomials of degree $[D_{2}],$ etc., to the ideal $I_{R_{d}}.$ The resulting ideal gives a 0-dimensional subscheme of $X,$ whose length, $\gamma_{\underline{p}},$ is independent of choices made and computable in \verb+Macaulay2+ \cite{M2}.

\section{Examples}
In this section we will compute the Segre class of subschemes of smooth projective toric varieties for some examples. This will be done using the algorithm given in Section \ref{ALG}, and \verb+Macaulay2+ \cite{M2} and \verb+Sage+ \cite{SAGE}. 

\subsection{A first example: Hirzebruch surfaces}
We here show how smoothly the algorithm of Eklund--Jost--Peterson generalizes to slightly more complicated toric varieties, namely Hirzebruch surfaces. Of course, surfaces are easier to treat than general varieties, since their subschemes are either divisors, zero-dimensional or of mixed dimension, and only the last case poses problems.
\begin{defi}
For each integer $e \geq 0,$ the $e$th Hirzebruch surface is the ruled surface over $\mathbb{P}^{1}$ defined by
\begin{equation*}
\Fe:= \mathbb{P}(\mathscr{O}_{\mathbb{P}^{1}} \oplus \mathscr{O}_{\mathbb{P}^{1}}(-e)).
\end{equation*} 
\end{defi}

\begin{remark}
Note that $\mathbb{F}_{0}$ is simply $\mathbb{P}^{1} \times \mathbb{P}^{1},$ while $\mathbb{F}_{1}$ is isomorphic to $\mathbb{P}^{2}$ blown up in a point. For $e>0$ there is a special curve on $\Fe;$ the only rational, irreducible curve with negative self-intersection number, and this number is $-e.$ Considered as a bundle over $\mathbb{P}^{1},$ the fibers of $\Fe$ are irreducible curves with self-intersection numbers 0, all rationally equivalent to each other.
\end{remark}

Now, let $\Sigma_e$ be the fan in $\mathbb{Z}^2$ with ray generators
\begin{displaymath}
v_1=\begin{bmatrix}1\\0 \end{bmatrix}, \quad v_2=\begin{bmatrix}0\\1 \end{bmatrix}, \quad v_3=\begin{bmatrix}-1\\e \end{bmatrix}, \quad v_4=\begin{bmatrix}0 \\-1 \end{bmatrix}
\end{displaymath}
corresponding to rays $\rho_i, 1 \leq i \leq 4.$ For each $e \geq 0$, the toric variety associated to $\Sigma_e$ is precisely the Hirzebruch surface $\Fe$. The torus invariant divisor associated to the ray $\rho_{i}$ is denoted by $D_{i}.$ The Picard group of $\Fe$ is generated by two such divisors whose classes are linearly independent. Choosing $\{[D_{1}],[D_{2}]\}$ as basis for the Picard group (i.e., the class of a fiber and the class of the exceptional curve), the intersection matrix is $\begin{bmatrix}0 & 1 \\ 1 & -e\end{bmatrix}.$ The Cox ring of $\Fe$ can be given by the variables $x_{0} = x_{\rho_{1}},x_{1} = x_{\rho_{3}},y_{0} = x_{\rho_{4}},y_{1} = x_{\rho_{2}}.$ 
The choice of basis for the the Picard group leads to a grading of the variables in the Cox ring by $(1,0), (1,0), (e,1)$ and $(0,1).$ The Cox ring $S = \mathbb{C}[x_{0},x_{1},y_{0},y_{1}]$ can therefore be described as
\begin{equation*}
S(a,b) := \bigoplus_{\substack{\alpha_{0}+\alpha_{1}+e\beta_{0} = a \\ \beta_{0}+\beta_{1} = b}} \mathbb{C}x_{0}^{\alpha_{0}}x_{1}^{\alpha_{1}}y_{0}^{\beta_{0}}y_{1}^{\beta_{1}}.
\end{equation*}
Let $B$ denote the irrelevant ideal of $\Fe,$ i.e., $B := (x_{0}y_{0},x_{0}y_{1},x_{1}y_{0},x_{1}y_{1}).$ 

Denote by $I \subseteq S$ the $B$-saturated ideal of a subscheme $Z \stackrel{j}{\hookrightarrow} \mathbb{F}_e$. We will provide a closed expression for $j_{\ast}s(Z,\Fe) \in A^{\ast}(\Fe)$ in terms of quantities which are either known or ``easily computable.'' Let $E:=[D_2]$ denote the class of the exceptional $(-e)$-curve on $\Fe,$ and let $F:=[D_1]$ denote the class of a fiber. By Theorem \ref{thm:Danilov} we have $A^{\ast}(\Fe) = \mathbb{Z}[E,F]/(E^2+eEF, F^2),$ and $j_{\ast}s(Z,\Fe)$ is a polynomial in $E$ and $F$; more precisely, we want to find the integers $\alpha, \beta, \gamma$ in the expression
\begin{equation*}
j_{\ast}s(Z,\Fe) = \alpha F + \beta E + \gamma EF, \quad (\alpha, \beta, \gamma) \in \mathbb{Z}^3.
\end{equation*}
The passage from the ideal $I$ to the identification of these three integers is not immediate, unless $Z$ is an effective divisor or a zero-dimensional scheme. In the first case, $Z$ is the zero scheme of a single polynomial $f \in S,$ of bidegree, say, $(a,b),$ and $$j_{\ast}s(Z,\Fe) = aF+bE+(b^2e-2ab)EF.$$ The bidegree $(a,b)$ can be computed using \verb+Macaulay2+ (input \verb+>degrees I+). In the second case, $j_{\ast}s(Z,\Fe) = d EF.$ Here, the integer $d$ can be obtained by computing the degree of the ideal $I$ in \verb+Macaulay2+ \cite{M2} (input \verb+>degree I+).

The remaining case is the one where $Z$ is assumed to be of mixed dimension. Then we have to use Theorem \ref{thm:main} and Proposition \ref{prop:formula} to compute the Segre class, and moving to this situation requires several manipulations. Suppose $I$ is generated by bihomogeneous polynomials $(g_0,\ldots, g_t)$, where $g_i$ has bidegree $(a_i, b_i)$. We start by recalling when a linear system on Hirzebruch surfaces is base-point free:

\begin{theorem}
If $a,b \geq 0,$ the linear system $|aD_{1} + bD_{2}|$ on $\mathbb{F}_{e}$ is base-point free if and only if $eb \leq a.$
\end{theorem}

\begin{proof}
This is \cite[Theorem V.2.17]{Hart:1977}, where his $f$ is our $[D_{1}]$ and his $C_{0}$ is our $[D_{2}].$
\end{proof}

Let $(a,b)$ be the apex of the cone $N(g_{0},\ldots,g_{t})$. In this case, a quick computation gives $a=\max_i\{e b_i\}-\min_i \{eb_i-a_i\}$ and $b=\max_i \{b_i\}$. Pick two general elements $f_1$ and $f_2$ in $I(a,b)$, where $I(a,b)$ denotes the bihomogeneous part of $I$ of bidegree $(a,b)$. We proceed in two steps. First, define the ideal $J_1:=((f_1):B^{\infty}),$ so that $V(J_{1})$ is a (possibly reducible) curve on $\Fe$. Let $I_{R_{1}}$ be the ideal $(J_{1}:I^{\infty})$. Then the subscheme $R_{1}:=V(I_{R_{1}})$ corresponds to $V(J_{1})$ with components from $Z$ removed, and it is either a divisor or empty. If $R_1$ is a divisor, then $I_{R_{1}}$ must be equal to $(f)$ for some bihomogeneous polynomial of bidegree, say, $(b^{(1)}_1,b^{(1)}_2)$, and we can, as previously, find the Segre class of $R_1$,
\begin{equation*}
j_{\ast}s(R_{1},\Fe) = b^{(1)}_1F+b^{(1)}_2E + ((b^{(1)}_2)^2e-2b^{(1)}_1b^{(1)}_2)EF.
\end{equation*}
If $R_{1}$ is empty, then the original scheme $Z$ was already a divisor, and the same formula holds with $b^{(1)}_1=b^{(1)}_2=0$. 
Hence, with $R_{1}$ the residual scheme to $Z$ in $V({J_{1}}),$ Proposition \ref{prop:formula} yields
\begin{equation*}
s_0=(a-b^{(1)}_1)F+(b-b^{(1)}_2)E.
\end{equation*}

Secondly, define $J_{2}:=((f_1,f_2):B^{\infty})$. With $I_{R_{2}} :=(J_{2}:I^{\infty})$, the subscheme $R_{2}:=V(I_{R_{2}})$ is either 0-dimensional or empty. Assume $R_{2}$ is 0-dimensional, then $j_{\ast}s(R_{2},\Fe)=b^{(2)}EF$. The integer $b^{(2)}$ can be obtained by computing the degree of $I_{R_{2}}$ in \verb+Macaulay2+ \cite{M2}. If $R_{2}$ is empty, the same holds with $b^{(2)}=0$. With $R_{2}$ the residual scheme to $Z$ in $V(J_{2})$, Proposition \ref{prop:formula} yields
\begin{eqnarray*}
s_1 & = & (aF+bE)^2-b^{(2)}EF-2(aF+bE)s_0.
\end{eqnarray*}

Combining the two expressions above, we get 
\begin{eqnarray*}
s_0 & = & (a-b^{(1)}_1)F+(b-b^{(1)}_2)E; \\
s_1 & = & (2ab^{(1)}_2+2bb^{(1)}_1-2ab+b^2e-2bb^{(1)}_2e-b^{(2)})EF.
\end{eqnarray*}
\noindent Thus, the Segre class of a subscheme of $\Fe$ is directly obtained from its ideal in the homogeneous coordinate ring.

\begin{example} \label{ex:hirzebruch}
Let $Z$ be the subscheme of $\mathbb{F}_1$ given by the $B$-saturated ideal $I=(x_1^2y_0^2+x_0^3x_1y_1^2,x_1y_0^2y_1^2+x_0^3y_1^4)$. The generators have bidegrees $(4,2)$ and $(3,4),$ thus the apex of the cone $N(g_{0},g_{1})$ is $(6,4).$ The choice of $\alpha=(6,4)$ gives us sufficient freedom to compute the necessary residual schemes and their classes, see Figure \ref{fig:cone}.

\begin{figure}
\begin{center}
\includegraphics[trim = 50mm 140mm 30mm 20mm, clip, width=8cm]{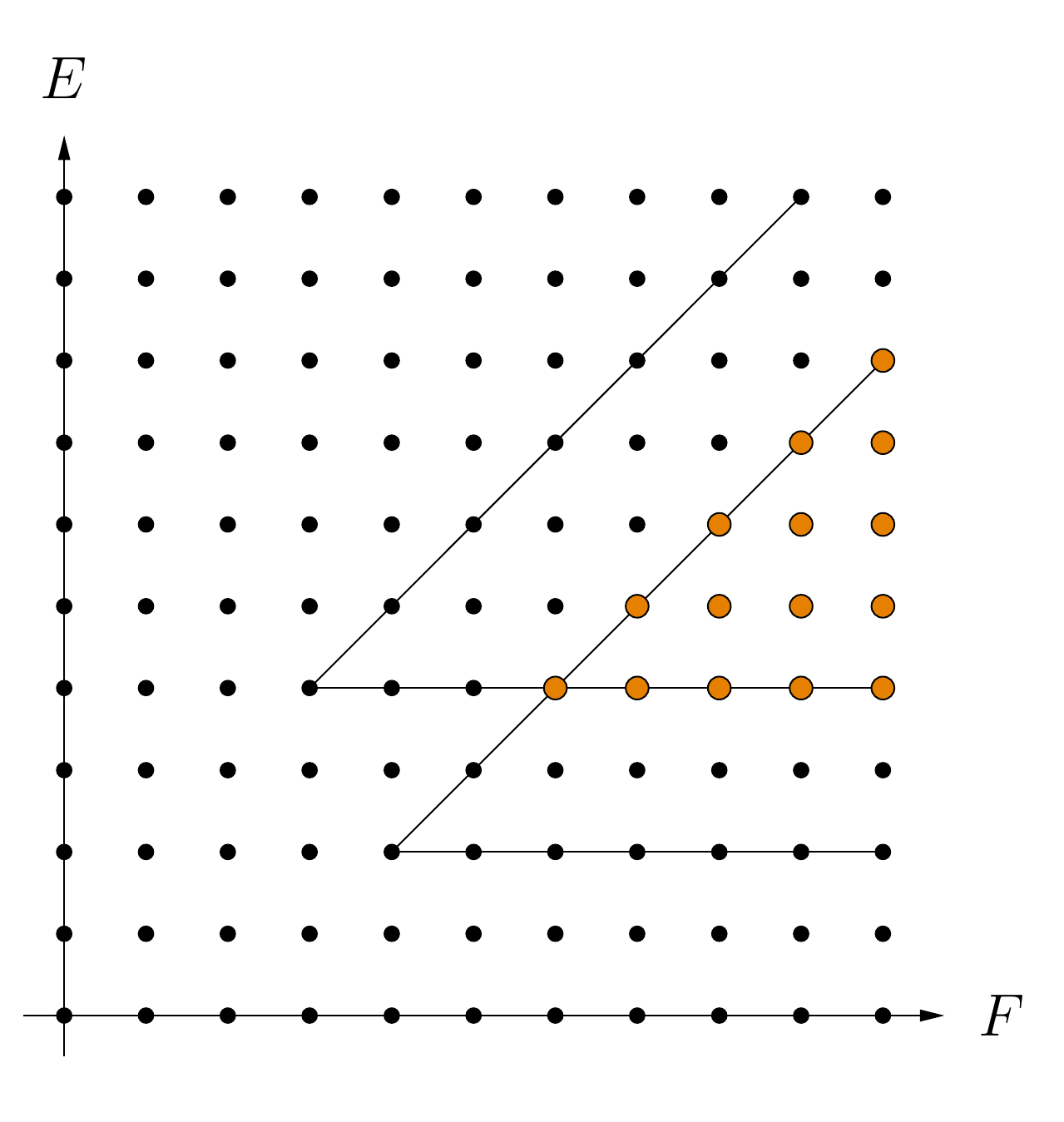}
\end{center}
\caption{The cone $N(g_0,g_1)$ in Example \ref{ex:hirzebruch}.}
\label{fig:cone}
\end{figure}

The computations can be performed in \verb+Macaulay2+ \cite{M2} with the following code:
\footnotesize
\begin{verbatimtab}
	S = QQ[x0,x1,y0,y1,Degrees=>{{1,0},{1,0},{1,1},{0,1}},Heft=>{1,1}]
	B = ideal(x0*y0,x1*y1,x0*y1,x1*y0)
	IS = ideal(x1^2*y0^2+x0^3*x1*y1^2,x1*y0^2*y1^2+x0^3*y1^4) 
	gensIS = flatten sort entries gens IS
	J = for i from 1 to 2 list sum(gensIS, g -> g*random({6,4}-degree(g),S))
	J1 = ideal(J_0)
	J2 = ideal(J_0,J_1)
	IR1 = saturate(saturate(J1,B),IS)
	IR2 = saturate(saturate(J2,B),IS)
	IR = degrees IR1
	IRL = IR_0
	b11 = IRL_0
	b12 = IRL_1
	b2 = degree IR2 
\end{verbatimtab}
\normalsize

The output of these computations is
\begin{eqnarray*}
(b_1^{(1)},b_2^{(1)}) &= & (3,2),\\
b_2&=&6.
\end{eqnarray*}

Using the acquired formulas above, the Segre classes of the subscheme $Z$ are
\begin{eqnarray*}
s_0&=&3F+2E,\\
s_1&=&-6EF.
\end{eqnarray*}

\end{example}

\subsection{Subschemes of $\mathbb{P}^1 \times \mathbb{P}^1 \times \mathbb{P}^1$}
In this example we will compute the Segre class of a subscheme of the toric variety $X:=\mathbb{P}^1 \times \mathbb{P}^1 \times \mathbb{P}^1$. The situation here is slightly more difficult than in the case of Hirzebruch surfaces because there is a greater variation of subschemes. On the other hand, the nef cone coincides with the effective cone, thus simplifying the choice of $\alpha.$ We include the general computation of the Segre class of a subscheme of $X$ to further illustrate our method. Last we show by an example how to find the Segre class of a particular subscheme. 

Note that $X$ is the toric variety associated to the fan $\Sigma$ given by the ray generators
\begin{displaymath}
v_1=\begin{bmatrix}1\\0\\0 \end{bmatrix}, \quad v_2=\begin{bmatrix}0\\1\\0 \end{bmatrix}, \quad v_3=\begin{bmatrix}0\\0\\1 \end{bmatrix}, \quad v_4=\begin{bmatrix}-1\\0 \\0 \end{bmatrix}, \quad v_5=\begin{bmatrix}0\\-1\\0 \end{bmatrix}, \quad v_6=\begin{bmatrix}0\\0\\-1\end{bmatrix},
\end{displaymath}
corresponding to rays $\rho_i, 1 \leq i \leq 6,$ and maximal cones
\begin{align*}
&\sigma_{1}=(\rho_1,\rho_2,\rho_3), \; \sigma_2=(\rho_2,\rho_3,\rho_4), \; \sigma_3=(\rho_3,\rho_4,\rho_5), \; \sigma_4=(\rho_1,\rho_3,\rho_5),\\
&\sigma_{5}=(\rho_1,\rho_2,\rho_6),\;\sigma_6=(\rho_2,\rho_4,\rho_6), \; \sigma_7=(\rho_4,\rho_5,\rho_6), \; \sigma_8=(\rho_1,\rho_5,\rho_6).
\end{align*}

The Cox ring of this variety is $$S(X) = \mathbb{C}[x_{0},x_{1},y_{0},y_{1},z_{0},z_1],$$ where the variables are can be trigraded in the following way,
\begin{eqnarray*}
\deg x_0 & =  (1,0,0)  = & \deg x_1,\\
\deg y_0 & =  (0,1,0)  = & \deg y_1,\\
\deg z_0 & =  (0,0,1)  = & \deg z_1.\\
\end{eqnarray*}
Hence, the trihomogeneous parts of $S$ can be described as, 
\begin{equation*}
S(a,b,c) := \bigoplus_{\substack{\alpha_0+\alpha_1=a \\ \beta_0+\beta_1 = b \\ \gamma_0+\gamma_1=c}} \mathbb{C}x_{0}^{\alpha_{0}}x_{1}^{\alpha_{1}}y_{0}^{\beta_{0}}y_{1}^{\beta_{1}}z_0^{\gamma_0}z_1^{\gamma_1}.
\end{equation*}
Observe that we have the irrelevant ideal $$B=(x_0y_0z_0,x_0y_0z_1,x_0y_1z_0,x_0y_1z_1,x_1y_0z_0,x_1y_0z_1,x_1y_1z_1,x_1y_1z_1).$$

The Chow ring of $X$ can be found by Theorem \ref{thm:Danilov},
\begin{equation*}
A^{\ast}(X) = \mathbb{Z}[D_1,D_2,D_3,D_4,D_5,D_6]/J,
\end{equation*}
where $J=(D_1D_4,D_2D_5,D_3D_6,D_1-D_4,D_2-D_5,D_3-D_6)$.
More precisely, for $0 \leq i \leq 3,$ $A^{i}(X)$ corresponds to polynomials of degree $i$ in $D_1,D_2,D_3,D_4,D_5,D_6.$ We trivially have $A^{0}(X) = \mathbb{Z}.$ By the relations given by $J$, $A^{1}(X)$ has three independent generators, for example $D_1$, $D_2$ and $D_3.$ Moreover, $A^{2}(X)$ can be generated by $D_1D_2$, $D_1D_3$ and $D_2D_3$, and everything else of degree 2 vanishes. In the same way, $A^{3}(X)$ is generated by $D_1D_2D_3$, and everything else of degree 3 vanishes.
Summing up, we have
\begin{eqnarray*}
A^{0}(X) & \cong & \mathbb{Z},\\
A^1(X) & \cong & D_1\mathbb{Z} \oplus D_2 \mathbb{Z} \oplus D_3 \mathbb{Z}, \\
A^2(X) & \cong & D_2D_3\mathbb{Z} \oplus D_1D_3 \mathbb{Z} \oplus D_1D_2 \mathbb{Z},\\
 A^{3}(X) & \cong & D_1D_2D_3 \mathbb{Z}.
\end{eqnarray*}

Let $Z$ be a subscheme of $X$ given by the $B$-saturated ideal $I\subset S$. Our aim is, as always, to obtain an expression for $j_{\ast}s(Z,X) \in A^{\ast}(X)$, with $j: Z \hookrightarrow X$ the inclusion. For clarity, we will consider all possible dimensions of $Z$ separately. The results are summarized in the following propositions.

\begin{proposition}
Let $Z$ be a $0$-dimensional subscheme of $\mathbb{P}^{1} \times \mathbb{P}^{1} \times \mathbb{P}^{1}$ determined by a homogeneous ideal $I$ in $S$, then its Segre class is $s_0=b_0 D_1D_2D_3$, where $b_0 \in \mathbb{Z}$ is the degree of $I$.
\end{proposition}

\begin{proposition}
Let $Z$ be a $1$-dimensional subscheme of $\mathbb{P}^{1} \times \mathbb{P}^{1} \times \mathbb{P}^{1}$ determined by a homogeneous ideal $I=(g_0,\ldots, g_t)$, where $g_i \in S(a_i,b_i,c_i)$. Set $$a:=\max \{a_i\}, b:=\max \{b_i\}, c:=\max \{c_i\}.$$ Then the Segre class of $Z$ is given by
\begin{eqnarray*}
s_0 & = & (2bc-b^{(2)}_1)D_2D_3+(2ac-b^{(2)}_2)D_1D_3+(2ab-b^{(2)}_3)D_1D_2,\\
s_1 & = & \bigl(3ab^{(2)}_1+3bb^{(2)}_2+3cb^{(2)}_3-b^{(3)}-12abc \bigr)D_1D_2D_3, 
\end{eqnarray*}
where the $b^{(d)}_i$'s are the coefficients of the classes of the residual schemes $R_d$.
\end{proposition}

\begin{proof}
 Let $f_1,f_2,f_3$ be general elements of $I(a,b,c)$. We now find expressions for $[R_d]$ for $d$ from $2$ to $3$.

$\underline{d=2}$. Let $J_2=((f_1,f_2):B^{\infty})$ and $R_2=V(J_2:I^{\infty})$. Then $R_2$ is purely $1$-dimensional (or empty) and $$[R_2]=b^{(2)}_1 D_2D_3+b^{(2)}_2 D_1D_3+ b^{(2)}_3 D_1D_2 \in A^2(X).$$ By Proposition \ref{prop:formula} we have, 
\begin{eqnarray*}
s_0 & = & (aD_1+bD_2+cD_3)^2 -(b^{(2)}_1 D_2D_3+b^{(2)}_2 D_1D_3+ b^{(2)}_3 D_1D_2)\\
	    & = & (2bc-b^{(2)}_1)D_2D_3+(2ac-b^{(2)}_2)D_1D_3+(2ab-b^{(2)}_3)D_1D_2. 
\end{eqnarray*}

$\underline{d=3}$. Let $J_3=((f_1,f_2,f_3):B^{\infty})$ and $R_3=V(J_3:I^{\infty})$. Then $R_3$ is $0$-dimensional (or empty) and $$[R_3]=b^{(3)} D_1D_2D_3 \in A^3(X).$$ By Proposition \ref{prop:formula} we have,

\begin{eqnarray*}
s_1 & = & (aD_1+bD_2+cD_3)^3 -3(aD_1+bD_2+cD_3)s_0-b^{(3)} D_1D_2D_3\\
	    & = & \bigl(3ab^{(2)}_1+3bb^{(2)}_2+3cb^{(2)}_3-b^{(3)}-12abc \bigr)D_1D_2D_3. 
\end{eqnarray*}
\end{proof}

\begin{proposition}\label{p1p1p1n2}
Let $Z$ be a $2$-dimensional subscheme of $\mathbb{P}^{1} \times \mathbb{P}^{1} \times \mathbb{P}^{1}$ determined by a homogeneous ideal $I=(g_0,\ldots, g_t)$, where $g_i \in S(a_i,b_i,c_i)$. Set $$a:=\max \{a_i\}, b:=\max \{b_i\}, c:=\max \{c_i\}.$$ Then the Segre class of $Z$ is given by

\begin{eqnarray*}
s_0 & = & (a-b^{(1)}_1)D_1+(b-b^{(1)}_2)D_2+(c-b^{(1)}_3)D_3,\\
s_1 & = & \bigl(2bb^{(1)}_3+2cb^{(1)}_2-2bc-b^{(2)}_1\bigr)D_2D_3 \\
            &   & + \bigl(2ab^{(1)}_3+2cb^{(1)}_1-2ac-b^{(2)}_2\bigr)D_1D_3 \\
	    &   & +\bigl(2ab^{(1)}_2+2bb^{(1)}_1-2ab-b^{(2)}_3\bigr)D_1D_2,\\
s_2 & = & 6abc -6bcb^{(1)}_1-6acb^{(1)}_2-6abb^{(1)}_3+3ab^{(2)}_1+3bb^{(2)}_2+3cb^{(2)}_3-b^{(3)},
\end{eqnarray*}
where the $b^{(d)}_i$'s are the coefficients of the classes of the residual schemes $R_d$.
\end{proposition}

\begin{proof}
Let $f_1,f_2,f_3$ be general elements of $I(a,b,c)$. We now find expressions for $[R_d]$ for $d$ from $1$ to $3$.

$\underline{d=1}$. Let $J_1=(f_1:B^{\infty})$ and $R_1=V(J_1:I^{\infty})$. Then $R_1$ is purely $2$-dimensional (or empty) and $$[R_1]=b^{(1)}_1 D_1+b^{(1)}_2 D_2+ b^{(1)}_3 D_3 \in A^1(X).$$ By Proposition \ref{prop:formula} we have,
\begin{eqnarray*}
s_0 & = & aD_1+bD_2+cD_3 -(b^{(1)}_1 D_1+b^{(1)}_2 D_2+ b^{(1)}_3 D_3)\\
	    & = & (a-b^{(1)}_1)D_1+(b-b^{(1)}_2)D_2+(c-b^{(1)}_3)D_3. 
\end{eqnarray*}

$\underline{d=2}$. Let $J_2=((f_1,f_2):B^{\infty})$ and $R_2=V(J_2:I^{\infty})$. Then $R_2$ is purely $1$-dimensional (or empty) and $$[R_2]=b^{(2)}_1 D_2D_3+ b^{(2)}_2 D_1D_3+b^{(2)}_3 D_1D_2 \in A^2(X).$$ By Proposition \ref{prop:formula} we have,

\begin{eqnarray*}
s_1 & = & (aD_1+bD_2+cD_3)^2 -(b^{(2)}_1 D_2D_3+ b^{(2)}_2 D_1D_3+b^{(2)}_3 D_1D_2)\\
 & &-2(aD_1+bD_2+cD_3)s_0\\
	    & = & \bigl(2bb^{(1)}_3+2cb^{(1)}_2-2bc-b^{(2)}_1\bigr)D_2D_3 \\
            &   & + \bigl(2ab^{(1)}_3+2cb^{(1)}_1-2ac-b^{(2)}_2\bigr)D_1D_3 \\
	    &   & +\bigl(2ab^{(1)}_2+2bb^{(1)}_1-2ab-b^{(2)}_3\bigr)D_1D_2.
\end{eqnarray*}

$\underline{d=3}$. Let $J_3=((f_1,f_2,f_3):B^{\infty})$ and $R_3=V(J_3:I^{\infty})$. Then $R_3$ is $0$-dimensional (or empty) and $$[R_3]=b^{(3)} D_1D_2D_3 \in A^3(X).$$ By Proposition \ref{prop:formula} we have,

\begin{eqnarray*}
s_2 & = & (aD_1+bD_2+cD_3)^3-3(aD_1+bD_2+cD_3)^2s_0\\
& &-3(aD_1+bD_2+cD_3)s_1-b^{(3)} D_1D_2D_3\\
& = & 6abc -6bcb^{(1)}_1-6acb^{(1)}_2-6abb^{(1)}_3+3ab^{(2)}_1+3bb^{(2)}_2+3cb^{(2)}_3-b^{(3)}.
\end{eqnarray*}
\end{proof}

Note that it is sufficient to take $\alpha = (\max \{a_{i}\}, \max \{b_{i}\}, \max \{c_{i}\}),$ because of the simple structure of the nef cone.

\begin{example}\label{ex:p1xp1xp1}
Let $Z$ be the subscheme given by the ideal $I=(x_0z_0^2,(y_0+y_1)z_0)$. Observe that we have $n=2$ and $(a,b,c)=(1,1,2)$. In \verb+Macaulay2+ we give the following input: 
\footnotesize
\begin{verbatimtab}
	S = QQ[x0,x1,y0,y1,z0,z1,
		Degrees=>{{1,0,0},{1,0,0},{0,1,0},{0,1,0},
		{0,0,1},{0,0,1}},Heft=>{1,1,1}]
	B = ideal(x0*y0*z0,x0*y0*z1,x0*y1*z0,
		x0*y1*z1,x1*y0*z0,x1*y1*z0,x1*y0*z1,x1*y1*z1)
	I = ideal(x0*z0^2,y0*z0+z0*y1)
	gensI = flatten sort entries gens I
	degI = degrees I
	transDegI = transpose degI
	len = length transDegI
	maxDegs = for i from 0 to len-1 list max transDegI_i
	J = for i from 1 to 3 list sum(gensI,
		g -> g*random(maxDegs-degree(g),S));
	J1 = ideal(J_0)
	J2 = ideal(J_0,J_1)
	J3 = ideal(J_0,J_1,J_2)
	IR1 = saturate(saturate(J1,B),I)
	IR2 = saturate(saturate(J2,B),I)
	IR3 = saturate(saturate(J3,B),I)
	IR = degrees IR1
	IRL = IR_0
	b11 = IRL_0
	b12 = IRL_1
	b13 = IRL_2
	b21 = degree(saturate((IR2+random({1,0,0},S),B)))
	b22 = degree(saturate((IR2+random({0,1,0},S),B)))
	b23 = degree(saturate((IR2+random({0,0,1},S),B)))
	b3 = degree IR3 
\end{verbatimtab}
\normalsize

The output is:
\begin{eqnarray*}
(b^{(1)}_1,b^{(1)}_2,b^{(1)}_3) & = & (1,1,1),\\
(b^{(2)}_1,b^{(2)}_2,b^{(2)}_3) & = & (1,2,1),\\
b^{(3)} & = & 2.
\end{eqnarray*}

We may then compute the Segre class $j_{\ast}s(Z,X)$ directly using Proposition \ref{p1p1p1n2}, yielding:
\begin{eqnarray*}
s_0 & = & D_3,\\
s_1 & = & D_2D_3+D_1D_2,\\
s_2 & = &-5D_1D_2D_3.
\end{eqnarray*}

\end{example}

\subsection{A third example}
We would like to illustrate the algorithm with a final example that would probably be an exhausting task to do by hand. Let $X$ be the toric variety associated to the fan $\Sigma$ given by the ray generators
\begin{displaymath}
v_0=\begin{bmatrix}1\\0\\0 \end{bmatrix}, \quad v_1=\begin{bmatrix}0\\1\\0 \end{bmatrix}, \quad v_2=\begin{bmatrix}0\\0\\1 \end{bmatrix}, \quad v_3=\begin{bmatrix}-1\\-1 \\0 \end{bmatrix}, \quad v_4=\begin{bmatrix}0\\0\\-1 \end{bmatrix},
\end{displaymath}
corresponding to rays $\rho_i, 0 \leq i \leq 4,$ and maximal cones
\begin{align*}
&\sigma_{1}=(\rho_0,\rho_1,\rho_2), \; \sigma_2=(\rho_1,\rho_2,\rho_3), \; \sigma_3=(\rho_0,\rho_2,\rho_3),\\
&\sigma_4=(\rho_0,\rho_1,\rho_4), \; \sigma_{5}=(\rho_1,\rho_3,\rho_4),\;\sigma_6=(\rho_0,\rho_3,\rho_4).
\end{align*}

Let $Z$ be the subscheme of $X$ given by the ideal $I=(x_1x_2,x_3x_4)$ in the coordinate ring of $X$. To compute its Segre class, we use the implementation of the algorithm in \verb+Macaulay2+ and \verb+Sage+, available for download at \linebreak \verb+http://sourceforge.net/projects/toricsegreclass+ together with instructions for use and more examples. In \verb+Macaulay2+ we give the following input:
\footnotesize
\begin{verbatimtab}
	load "toricSegreClass.m2"
	Rho = {{1,0,0},{0,1,0},{0,0,1},{-1,-1,0},{0,0,-1}}
	Sigma = {{0,1,2},{1,2,3},{0,2,3},{0,1,4},{1,3,4},{0,3,4}}
	XT = normalToricVariety(Rho,Sigma)
	deg=matrix{{1,1,0,1,0},{0,0,1,0,1}}
	X=normalToricVariety(rays XT, max XT, WeilToClass => deg)
	S=ring X
	I=ideal(x_1*x_2,x_3*x_4)
	toricSegreClass(X,I)
\end{verbatimtab}
\normalsize

In \verb+Sage+, we include the output from \verb+Macaulay2+ and write:
\footnotesize
\begin{verbatimtab}
	load "toricSegreClass.sage"
	integerLists = [[[2, 0, 0, 0, 0], [1, 1, 0, 0, 0],
			[0, 2, 0, 0, 0], [1, 0, 1, 0, 0],
			[0, 1, 1, 0, 0], [0, 0, 2, 0, 0],
			[1, 0, 0, 1, 0], [0, 1, 0, 1, 0],
			[0, 0, 1, 1, 0], [0, 0, 0, 2, 0],
			[1, 0, 0, 0, 1], [0, 1, 0, 0, 1],
			[0, 0, 1, 0, 1], [0, 0, 0, 1, 1],
			[0, 0, 0, 0, 2]], 
			[[1, 0, 0, 0, 0], [0, 1, 0, 0, 0],
			[0, 0, 1, 0, 0], [0, 0, 0, 1, 0],
			[0, 0, 0, 0, 1]], [[0, 0, 0, 0, 0]]]
	gammaLists =  [[1, 1, 1, 1, 1, 0, 1, 1, 1, 1, 1, 1, 0, 1, 0], 
			[2, 2, 1, 2, 1], [0]]
	degAlpha = [1, 1]
	fan = Fan(cones=[(0,1,2),(1,2,3),(0,2,3),(0,1,4),(1,3,4),(0,3,4)], 
	rays=[(1,0,0),(0,1,0),(0,0,1),(-1,-1,0),(0,0,-1)])
	X=ToricVariety(fan)
	X.inject_variables()
	I=ideal(z1*z2,z3*z4)
	s = toricSegreClass(fan, I, degAlpha, integerLists, gammaLists)
	print s
\end{verbatimtab}
\normalsize

\noindent The output is the cycle class \verb+( 12 | -2, 3 | 0, 0 | 0 )+.

\section{Algorithmic implementation}\label{ALG}
In this section we present the algorithm in pseudo-language. A preliminary implementation in \verb+Sage+ and \verb+Macaulay2+ is available online at \linebreak \verb+http://sourceforge.net/projects/toricsegreclass+. Note that intersection theory for toric varieties hopefully will be implemented in \verb+Macaulay2+ in the near future. At that point, all computations can be done using only \verb+Macaulay2+. At the time of writing, intersection theory for toric varieties is only implemented in \verb+Sage+, while some of the ring-theoretic operations call for \verb+Macaulay2+.

\smallskip
\noindent
{\textbf{Input:}} The fan of a smooth toric variety $X$ and an ideal $I$ 
with non-zero multi-homogeneous generators $g_0,\ldots,g_t$ in the coordinate ring of $X$.

\noindent {\textbf{Output:}} The Segre class of the subscheme of $X$ defined by $I$. 
\smallskip

\noindent {\textbf{Part I - Setup.}}
\begin{enumerate}

\item Let $\Sigma$ be a smooth normal fan of a lattice polytope in $\mathbb{Z}^k$ generated by $r$ rays.
\item Let $X$ be the smooth projective toric variety corresponding to $\Sigma$. 
\item Let $k:=\dim(X)$.
\item Compute $S=\textnormal{Cox}(X)$, the coordinate ring of $X$. 
\item For $i=1$ to $r$, let $x_i$ be the variables of $S$. 
\item For $i=1$ to $r$ let $D_i \in \textnormal{Div}(X)$ denote the corresponding divisor of $x_i$.
\item Let $\nu :=r-k$. For $i=1$ to $r$, store the multidegree of $x_i$,  $\underline{m}_i=\{m_j\}_{j=1}^{\nu}$.
\item Compute the irrelevant ideal $B$ of $X$.
\item Compute $A$, the Chow group of $X,$ and $N,$ the nef cone of $X.$
\end{enumerate}

\noindent {\textbf{Part II - Finding } $[R_d]$ \textbf{for } $k-n \leq d \leq k.$}
\begin{enumerate}
\item Let $Z \subset X$ be the subscheme defined by $I$ and compute $n :=\dim(Z)$.
\item Choose $\alpha$ as the apex of the cone $\textnormal{N}(g_{0},\ldots,g_{t}).$
\item Pick random elements $f_1, \ldots, f_k \in I(\alpha)$.
\item For $d=k-n$ to $k$ do
\begin{enumerate}
\item Let $J_d=((f_1,\ldots, f_d):B^{\infty})$.
\item Compute $I_{R_d}=(J_d:I^{\infty})$.
\item Let $m:=k-d=\textnormal{dim } R_d$, for $R_d$ the subscheme of $X$ defined by $I_{R_d}$.
\item Let $h_m :=\textnormal{rk } A_m(X)$, and let $\omega^{(d)}_1, \ldots, \omega^{(d)}_{h_m}$ denote the generators of $A_m(X)$.
\item Let $[R_d]=\sum_{i=1}^{h_m} b^{(d)}_i \cdot \omega^{(d)}_{i}$. Find $\geq h_m$ linear equations with variables $b_i$ in the following way. 
For all sets of natural numbers $\underline{p}=\{p_1, \ldots, p_{r}\}$ such that $\sum_{i=1}^r p_i=m$ do:
\begin{enumerate}
\item for $i=1, \ldots, r$ let $P_{i,j} \in S, j \in \{1,\ldots,p_{i}\},$ be $p_i$ random polynomials of multidegree $\underline{m}_i$. Compute the length $\gamma_{\underline{p}}$ of the 0-dimensional subscheme determined by $I_{R_d}+\sum_{i=1}^{r}\sum_{j=1}^{p_i} (P_{i,j}).$
\item let $\beta_{i,\underline{p}}$ be the degree of the intersection product $\omega^{(d)}_i \cdot \prod_{j=1}^{r}D_j^{p_j}$ for $i=1$ to $h_m$.
\end{enumerate}
Solve the linear system of equations $\left\{\sum_{i=1}^{h_m}\beta_{i,\underline{p}}\cdot b^{(d)}_i = \gamma_{\underline{p}}\right\}.$
\item Store $[R_d]$.
\end{enumerate}
\end{enumerate}

\noindent {\textbf{Part III - Finding } $j_{\ast}s(Z,X).$}
\begin{enumerate}
\item Consider the list $\alpha$ as a class in $A^1(X),$ through the isomorphism $A^{1}(X) \cong \textnormal{Pic }(X) \cong \mathbb{Z}^{r-k}.$ For $i=0$ to $n$ find $s_i$ by computing
\begin{eqnarray*}
s_{0} & = & \alpha^{k-n} - [R_{k-n}] \\
s_{i} & = & \alpha^{i+k-n} - [R_{i+k-n}] - \sum_{j=0}^{i-1} {i+k-n \choose i-j}\alpha^{i-j}s_{j}, \text{ for all } i \geq 1.
\end{eqnarray*}
\item Return $s := \sum_{i=0}^{n}s_i,$ the Segre class $j_{\ast}s(Z,X).$
\end{enumerate}


\bibliographystyle{hacm}
\bibliography{bib2}

\end{document}